\documentclass{amsart}

\usepackage{graphicx}

\newtheorem{theorem}{Theorem}[section]

\newtheorem{corollary}[theorem]{Corollary}

\theoremstyle{definition}

\theoremstyle{remark}
\newtheorem{remark}[theorem]{Remark}

\numberwithin{equation}{section}

\begin{document}

\title{A new look at Bernoulli's inequality}

%    Information for first author
\author{Rui A. C. Ferreira}
%    Address of record for the research reported here
\address{Grupo F\'isica-Matem\'atica, Faculdade de Ci\^encias, Universidade de Lisboa, Av. Prof. Gama Pinto, 2, 1649-003 Lisboa, Portugal.}
%    Current address
\curraddr{Grupo F\'isica-Matem\'atica, Faculdade de Ci\^encias, Universidade de Lisboa, Av. Prof. Gama Pinto, 2, 1649-003 Lisboa, Portugal.}
\email{raferreira@fc.ul.pt}
%    \thanks will become a 1st page footnote.
\thanks{The author was supported by the ``Funda\c{c}\~ao para a Ci\^encia e a Tecnologia (FCT)" through the program ``Investigador FCT" with reference IF/01345/2014.}

%    General info
\subjclass[2000]{Primary 26D15; Secondary 26A33}

%\date{January 1, 2001 and, in revised form, June 22, 2001.}

%\dedicatory{This paper is dedicated to our advisors.}

\keywords{Bernoulli's inequality, Discrete Fractional Calculus}

\begin{abstract}
In this work, a generalization of the well known Bernoulli inequality is obtained by using the theory of discrete fractional calculus. As far as we know our approach is novel.
\end{abstract}

\maketitle

\section{Introduction}

In classical analysis the following inequality is attributed to Bernoulli: for a real number $x>-1$ and a nonnegative integer $n$, it holds:

\begin{equation}\label{Bern}
(1+x)^n \geq 1+nx.
\end{equation}
One can find in the literature several (elementary) different proofs of inequality \eqref{Bern} (see e.g. \cite{Alfaro,Nelsen}). Moreover, various generalizations were also obtained throughout the years (cf. \cite{Mit}) as well as different kinds of applications (see e.g. \cite{Klen}).

In this work we obtain an inequality that generalizes \eqref{Bern} in a completely different direction than the ones mentioned before. The reasoning is that we use the theory of discrete fractional calculus \cite{Goodrich}, in particular, the (delta) Riemann--Liouville fractional operators which were introduced by Miller and Ross in 1988 \cite{Mill0} and for which real developments happened only in the past eight years. Therefore, we actually believe our main results to be new and obtained following a novel procedure.

In order to accomplish our desires we need to further develop the theory of linear fractional difference equations, which was initially started in the work \cite{Atici} and generalized afterwards in \cite{Ferreira}. More specifically, we solve explicitly the IVP\footnote{This result alone might obviously be used by researchers in other contexts.},
\begin{align*}
(\Delta^\nu_{a+\nu-1}
x)(t)&=y(t+\nu-1)x(t+\nu-1)+z(t+\nu-1),\ t\in\{a,a+1,a+2,\ldots\},\\
x(a+\nu-1)&=x_{a+\nu-1},
\end{align*}
and, after deducing some of its important consequences, we use a recent (comparison) result of \cite{Jia} to deduce our Bernoulli-type inequality.

This paper is organized as follows: In Section 2 we provide the reader some
background on the discrete fractional calculus theory. In Section 3 we present our achievements.

\section{Preliminaries on Discrete Fractional Calculus}\label{sec0}

In this section we introduce the reader to basic concepts and
results about discrete fractional calculus (the monograph \cite{Goodrich}, particularly Chapter 2, could be useful to that matter).

Throughout this work and as usual we assume that empty sums and
products equal 0 and 1, respectively.

The power function is defined by
\begin{align*}
x^{(y)}&=\frac{\Gamma(x+1)}{\Gamma(x+1-y)},\mbox{ for}\ x,x-y\in\mathbb{R}\backslash\{\ldots,-2,-1\},\\
x^{(y)}&=0,\mbox{ for } x\notin{\mathbb{Z}^-}\mbox{ and } x-y\in{\mathbb{Z}^-}.
\end{align*}

For $a\in\mathbb{R}$ and $0<\nu\leq 1$ we define the set
$\mathbb{N}_a=\{a,a+1,a+2,\ldots\}$.
Also, we use the notation $\sigma(s)=s+1$ for the shift operator
and $(\Delta f)(t)=f(t+1)-f(t)$ to the forward difference
operator.

For a function $f:\mathbb{N}_a\rightarrow\mathbb{R}$, the
\emph{discrete fractional sum of order $\nu\geq 0$} is defined
as
\begin{align}
(\Delta^{0}_a f)(t)&=f(t),\quad t\in\mathbb{N}_a,\nonumber\\
(\Delta^{-\nu}_a f)(t)&=\frac{1}{\Gamma(\nu)}\sum_{s=a}^{t-\nu}(t-\sigma(s))^{(\nu-1)}f(s),\quad
t\in\mathbb{N}_{a+\nu-1},\ \nu>0.\label{def0}
\end{align}
\begin{remark}
Note that the operator $\Delta^{-\nu}_a$ with $\nu>0$ maps
functions defined on $\mathbb{N}_a$ to functions defined on
$\mathbb{N}_{a+\nu-1}$. Also observe that if $\nu=1$, then we
get the summation operator:
$$(\Delta^{-1}_a f)(t)=\sum_{s=a}^{t-1}f(s).$$
\end{remark}

The \emph{discrete fractional derivative of order
$\nu\in(0,1]$} is defined by
$$(\Delta^{\nu}_a f)(t)=(\Delta\Delta^{-(1-\nu)}_af)(t),\quad t\in\mathbb{N}_{a+\nu-1}.$$
\begin{remark}
Note that if $\nu=1$, then the fractional derivative is just
the forward difference operator.
\end{remark}

\section{Main Results}\label{sec1}

This section is devoted in great part to deduce our generalized Bernoulli's inequality.

In \cite{Ferreira} it was shown that, for $t\in\mathbb{N}_{a+\nu-1}$,
\begin{equation}\label{eq5}
x(t)=\frac{(t-a)^{(\nu-1)}}{\Gamma(\nu)}x_{a+\nu-1}+\frac{1}{\Gamma(\nu)}\sum_{s=a}^{t-\nu}(t-\sigma(s))^{(\nu-1)}f(s+\nu-1,x(s+\nu-1)),
\end{equation}
is the solution of the following nonlinear fractional difference
initial value problem:
\begin{align*}
(\Delta^\nu_{a+\nu-1}
x)(t)&=f(t+\nu-1,x(t+\nu-1)),\quad t\in\{a,a+1,a+2,\ldots\},\\
x(a+\nu-1)&=x_{a+\nu-1}.
\end{align*}
For our purposes we need the solution of \eqref{eq5} in the particular case when
$f(t,x)=y(t)x+z(t)$. The next result is obtained following the same procedure as for the case $f(t,x)=y(t)x$ done in \cite{Ferreira} and, therefore, we leave the details of its proof to the reader. 

\begin{theorem}\label{thm7}
Let $a\in\mathbb{R}$ and $\nu\in(0,1]$. Suppose that
$y:\mathbb{N}_{a+\nu-1}\rightarrow\mathbb{R}$ is a function. Define
an operator $T$ by
\begin{align*}
(T_y^0f)(t)&=f(t),\\
(T_y^1f)(t)&=(T_yf)(t)=\left(\Delta^{-\nu}_ay(s+\nu-1)f(s+\nu-1)\right)(t),\\
(T^{k+1}_yf)(t)&=(T_yT_y^{k})(t),\quad k\in\mathbb{N}^1,
\end{align*}
for $t\in\mathbb{N}_{a+\nu-1}$. Then, the function
\begin{equation}\label{eq0}
x(t)=\sum_{k=a}^{t-(\nu-1)}\left[\frac{x_{a+\nu-1}}{\Gamma{(\nu)}}\left(T^{k-a}_y(s-a)^{(\nu-1)}\right)(t)+\left(T^{k-a}_y\Delta^{-\nu}_a z(s+\nu-1)\right)(t)\right],\
\end{equation}
is the solution of the summation equation
$$x(t)=\frac{(t-a)^{(\nu-1)}}{\Gamma(\nu)}x_{a+\nu-1}+\left(\Delta^{-\nu}_a[y(s+\nu-1)x(s+\nu-1)+z(s+\nu-1)]\right)(t),$$
for all $t\in\mathbb{N}_{a+\nu-1}$.
\end{theorem}

\begin{remark}
In Theorem \ref{thm7}, the notation 
$$\left(\Delta^{-\nu}_ay(s+\nu-1)f(s+\nu-1)\right)(t),$$
stands for $\left(\Delta^{-\nu}_a g\right)(t)$ where $g(s)=y(s+\nu-1)f(s+\nu-1)$.
\end{remark}

Let us now introduce some notation. We define a function $E:\mathbb{N}_{a+\nu-1}\times\mathbb{R}^5\rightarrow\mathbb{R}$ (it can be thought of as a discrete Mittag--Leffler function) by:

\begin{equation*}
E(t,a,\nu,\beta,c,\lambda)=\sum_{k=a}^{t-(\nu-1)}\frac{c^{k-a}}{\Gamma{((k-a)\nu+\beta)}}(t-\lambda+(k-a)(\nu-1))^{((k-a)\nu+\beta-1)},
\end{equation*}
whenever the right hand side makes sense.

\begin{corollary}\label{cor0}
If $y(t)=c$ for some $c\in\mathbb{R}$ and all
$t\in\mathbb{N}_{a+\nu-1}$ in Theorem \ref{thm7}, then the solution
given by \eqref{eq0} is
\begin{equation*}
x(t)=x_{a+\nu-1}E(t,a,\nu,\nu,c,a)
+\sum_{r=a}^{t-\nu}E(t,a,\nu,\nu,c,\sigma(r))z(r+\nu-1),\quad t\in\mathbb{N}_{a+\nu-1}.
\end{equation*}
\end{corollary}

It is pertinent to formulate the following consequence of Corollary \ref{cor0}.

\begin{corollary}
Suppose that the function $z$ is a constant equal to $K$ in Corollary \ref{cor0}. Then,
\begin{equation*}
x(t)=x_{a+\nu-1}E(t,a,\nu,\nu,c,a)
+KE(t,a,\nu,\nu+1,c,a),\quad t\in\mathbb{N}_{a+\nu-1}.
\end{equation*}
\end{corollary}

\begin{proof}
Let us first note that (cf. \cite[Theorem 1.8]{Goodrich}): $$\Delta_t(s-t)^{(r)}=-r(s-\sigma(t))^{(r-1)},\ s, r\in\mathbb{R}.$$
Moreover,
$$\sum_{k=a}^{b-1}\Delta f(k)=f(b)-f(a).$$
Therefore,
\begin{align*}
    &\sum_{r=a}^{t-\nu}E(t,a,\nu,\nu,c,\sigma(r))z(r+\nu-1)\\
    &=K\sum_{r=a}^{t-\nu}\sum_{k=a}^{t-(\nu-1)}\frac{c^{k-a}}{\Gamma{((k-a)\nu+\nu)}}(t-\sigma(r)+(k-a)(\nu-1))^{((k-a)\nu+\nu-1)}\\
    &=K\sum_{k=a}^{t-(\nu-1)}\frac{c^{k-a}}{\Gamma{((k-a)\nu+\nu)}}\sum_{r=a}^{t-(\nu-1)-1}(t-\sigma(r)+(k-a)(\nu-1))^{((k-a)\nu+\nu-1)}\\
    &=K\sum_{k=a}^{t-(\nu-1)}\frac{c^{k-a}}{\Gamma{((k-a)\nu+\nu)}((k-a)\nu+\nu)}\\
    &\hspace{1cm}\cdot\sum_{r=a}^{t-(\nu-1)-1}-((k-a)\nu+\nu)(t-\sigma(r)+(k-a)(\nu-1))^{((k-a)\nu+\nu-1)}\\
    &=K\sum_{k=a}^{t-(\nu-1)}\frac{c^{k-a}}{\Gamma{((k-a)\nu+\nu+1)}}[-(t-r+(k-a)(\nu-1))^{((k-a)\nu+\nu)}]_{r=a}^{r=t-(\nu-1)}\\
     &=K\sum_{k=a}^{t-(\nu-1)}\frac{c^{k-a}}{\Gamma{((k-a)\nu+\nu+1)}}(t-a+(k-a)(\nu-1))^{((k-a)\nu+\nu)}\\
     &=KE(t,a,\nu,\nu+1,c,a),
\end{align*}
and the proof is done.
\end{proof}
We now need to address the question of the sign of the function $E(t,a,\nu,\nu,c,a)$. First we note that it is the solution of the fractional IVP:
\begin{align*}
(\Delta^\nu_{a+\nu-1}
x)(t)&=cx(t+\nu-1),\quad t\in\{a,a+1,a+2,\ldots\},\ 0<\nu\leq 1,\\
x(a+\nu-1)&=1.
\end{align*}
Let us now recall a recent result proved by Jia \emph{et.al}:

\begin{theorem}\cite[Theorem 4.2.]{Jia}
Assume $c_1(t)\geq c_2(t)\geq -\nu$, $0<\nu<1$ and $x(t)$, $y(t)$ are solutions of the equations
$$(\Delta^\nu_{a+\nu-1}
x)(t)=c_1(t)x(t+\nu-1),\quad  t\in\{a,a+1,a+2,\ldots\}$$
and
$$(\Delta^\nu_{a+\nu-1}
y)(t)=c_2(t)y(t+\nu-1),\quad  t\in\{a,a+1,a+2,\ldots\},$$
respectively, satisfying $x(a+\nu-1)\geq y(a+\nu-1)>0$. Then,
$$x(t)\geq y(t),\quad t\in\mathbb{N}_{a+\nu-1}.$$
\end{theorem}
A closer look to the proof of the previous theorem permits us to conclude immediately that $y(a+\nu-1)$ might be equal to zero and the result still remains. Moreover, the representation for the Riemann--Liouville fractional difference,
\begin{equation}\label{op}
(\Delta^\nu_a f)(t)=\frac{1}{\Gamma(-\nu)}\sum_{k=a}^{t+\nu}(t-\sigma(k))^{(-\nu-1)}f(k),\ t\in\mathbb{N}_{a-\nu+1},
\end{equation}
is used to prove the result and, hence, the restriction $\nu\neq 1$. Nevertheless, M. Holm showed in \cite[Theorem 2.2.]{Holm} the continuity of the fractional difference operator \eqref{op} with respect to $\nu$ and, therefore, one may also consider $\nu=1$. In conclusion, for a real number $a$, $0<\nu\leq 1$ and $c\geq -\nu$:

\begin{equation}\label{pos}
E(t,a,\nu,\nu,c,a)\geq 0,\quad t\in \mathbb{N}_{a+\nu-1}.
\end{equation}

We are now ready to prove our main result:

\begin{theorem}(Generalized Bernoulli inequality)
Let $\nu\in(0,1]$, $c\in[-\nu,\infty)$ and $a\in\mathbb{R}$. Then, the following inequality holds:
\begin{equation}\label{GBer}
 cE(t,a,\nu,\nu+1,c,a)\geq c\frac{(t-a)^{(\nu)}}{\Gamma(\nu+1)},\quad t\in\mathbb{N}_{a+\nu-1}.
\end{equation}
\end{theorem}

\begin{proof}
Let $c\geq -\nu$ and $x:\mathbb{N}_{a+\nu-1}\rightarrow\mathbb{R}$ be the function defined by:
$$x(t)=c\frac{(t-a)^{(\nu)}}{\Gamma(\nu+1)}.$$
Then $x(a+\nu-1)=0$ and $(\Delta^\nu_{a+\nu-1} x)(t)=(\Delta^\nu_{a+\nu} x)(t)=c$ by \cite[Theorem 2.40]{Goodrich}. Therefore,
$$cx(t+\nu-1)+c=c^2\frac{(t+\nu-1-a)^{(\nu)}}{\Gamma(\nu+1)}+c\geq (\Delta^\nu_{a+\nu-1} x)(t).$$
Define the function $m$ by:
$$m(t+\nu-1)=cx(t+\nu-1)+c-(\Delta^\nu_{a+\nu-1} x)(t),$$
which is nonnegative. By Corollary \ref{cor0} we get (note that $x_a=0$)
$$x(t)=\sum_{r=a}^{t-\nu}E(t,a,\nu,\nu,c,\sigma(r))(c-m(r+\nu-1)).$$
Hence,

\begin{align*}
 x(t)&=cE(t,a,\nu,\nu+1,c,a)-\sum_{r=a}^{t-\nu}m(r+\nu-1)\\
    &\hspace{.5cm}\cdot\sum_{k=a}^{t-(\nu-1)}\frac{c^{k-a}}{\Gamma{((k-a)\nu+\nu)}}(t-\sigma(r)+(k-a)(\nu-1))^{((k-a)\nu+\nu-1)}\\
     &=cE(t,a,\nu,\nu+1,c,a)-\sum_{r=a}^{t-\nu}m(r+\nu-1)\\
     &\hspace{.5cm}\cdot\sum_{k=a}^{t-\sigma(r)+a-(\nu-1)}\frac{c^{k-a}}{\Gamma{((k-a)\nu+\nu)}}(t-\sigma(r)+(k-a)(\nu-1))^{((k-a)\nu+\nu-1)}\\
    &=cE(t,a,\nu,\nu+1,c,a)-\sum_{r=0}^{t-a-\nu}m(r+a+\nu-1)\\
    &\hspace{.5cm}\cdot\sum_{k=a}^{t-\sigma(r)-(\nu-1)}\frac{c^{k-a}}{\Gamma{((k-a)\nu+\nu)}}(t-\sigma(r)-a+(k-a)(\nu-1))^{((k-a)\nu+\nu-1)}\\
   &=cE(t,a,\nu,\nu+1,c,a)-\sum_{r=0}^{t-a-\nu}m(r+a+\nu-1)E(t-\sigma(r),a,\nu,\nu,c,a).
   \end{align*}

Finally, by \eqref{pos} we conclude that
$$x(t)\leq cE(t,a,\nu,\nu+1,c,a),$$
which is equivalent to
$$c\frac{(t-a)^{(\nu)}}{\Gamma(\nu+1)}\leq cE(t,a,\nu,\nu+1,c,a).$$
The proof is done.
\end{proof}

We finish this work showing that inequality \eqref{GBer} truly generalizes the Bernoulli inequality, i.e. when we let $\nu=1$ (and $a=0$) in \eqref{GBer}, then we get \eqref{Bern}:
\begin{align*}
    & cE(t,0,1,2,c,0)\geq c\frac{t^{(1)}}{\Gamma(2)}\\
    & \Leftrightarrow \sum_{k=0}^t\frac{c^{k+1}}{\Gamma(k+2)}t^{(k+1)}\geq ct\\
    & \Leftrightarrow -1+\sum_{k=0}^t\frac{c^{k}}{\Gamma(k+1)}t^{(k)}\geq ct\\
    & \Leftrightarrow (1+c)^t\geq 1+ct,\quad t\in\mathbb{N}_0,
\end{align*}
where the last equivalency follows from \cite[Remark 3.7.]{Ferreira}.

\section*{Acknowledgements}
The author would like to thank the referees for their careful reading of the manuscript and their suggestions that undoubtedly contributed to its final draft.

\bibliographystyle{amsplain}

\end{document}